\documentclass[reqno]{amsart}

\usepackage{amsmath}
\usepackage{amsthm}
\usepackage{amsfonts}
\usepackage{amssymb}
\usepackage{mathrsfs}
\usepackage{enumerate}
\usepackage[usenames,dvipsnames]{xcolor}
\usepackage{graphicx}
\usepackage{tikz,caption}
\usetikzlibrary{arrows,shapes,trees}
\usepackage{verbatim, hyperref}
\hypersetup{linkcolor=blue, citecolor=red, colorlinks=true}

\newtheorem{thm}{Theorem}
\newtheorem{lem}[thm]{Lemma}

\theoremstyle{definition}
\newtheorem{defn}[thm]{Definition}

\newtheorem{prop}[thm]{Proposition}

\theoremstyle{remark}

\theoremstyle{definition}
\newtheorem{exmp}[thm]{Example}

\newtheorem{rmk}[thm]{Remark}

\newtheorem{prob}{Problem}

\providecommand{\R}[0]{\mathbb{R}}

\providecommand{\N}[0]{\mathbb{N}}
\providecommand{\C}[0]{\mathbb{C}}

\newcommand{\bdf}{\begin{defn}}
\newcommand{\edf}{\end{defn}}

\newcommand{\brm}{\begin{rmk}}
\newcommand{\erm}{\end{rmk}}

\newcommand{\beq}{\begin{equation}}
\newcommand{\eeq}{\end{equation}}

\newcommand{\barr}[1]{\begin{array}{#1}}
\newcommand{\earr}{\end{array}}

\newtheorem*{thmmain}{Theorem \ref{thm:main}}

\DeclareMathOperator{\Arg}{Arg}
\def\Re{\mathop{\text{\rm Re}}}
\def\Im{\mathop{\text{\rm Im}}}

\newcommand{\lr}[1]{\left(#1\right)}
\newcommand{\br}[1]{\left[#1\right]}
\newcommand{\set}[1]{\left\{#1\right\}}

\newcommand{\disp}{\displaystyle}

\numberwithin{equation}{section}
\numberwithin{figure}{section}

\author{R. Bates, R. Yoshida}
\date{\today}
\title{Quadratic Hyperbolicity Preservers \& \mbox{Multiplier Sequences}}

\begin{document}
\setcounter{page}{1}
\pagenumbering{roman}

\begin{abstract}
It is known (see \cite[Br\"and\'en, Lemma 2.7]{B10}) that a necessary condition for $T:=\sum Q_k(x) D^k$ to be hyperbolicity preserving is that $Q_k(x)$ and $Q_{k-1}(x)$ have interlacing zeros. We characterize all quadratic linear operators, as a consequence we find several classes of $P_n$-multiplier sequence. 
\end{abstract}

\maketitle
\thispagestyle{empty}




\section{Introduction}

\setcounter{page}{1}
\pagenumbering{arabic}

It is well known (see {\cite{peet},\cite[p. 32]{piotr}}) that if $T$ is any linear operator defined on the space of real polynomials, $R[x]$, then there is a sequence of real polynomials, $\{Q_k(x)\}$, such that 
\begin{equation} \label{eq:linear-operator}
T=\sum Q_k(x) D^k,\ \mbox{where } D=\frac{d}{dx}.
\end{equation}
Our investigation involves such operators that act on polynomials, in particular, we are interested in polynomials with the following property.

\bdf
A polynomial $f(x)\in\R[x]$ whose zeros are all real is said to be \emph{hyperbolic}. Following the convention of G.\ P\'olya and J.\ Schur \cite[p.89]{PS}, the constant $0$ is also deemed to be hyperbolic. 
\edf

\bdf 
A linear operator $T:\R[x]\to\R[x]$ is said to \emph{preserve hyperbolicity} (or $T$ is a \emph{hyperbolicity preserver}) if $T[f(x)]$ is a hyperbolic polynomial, whenever $f(x)$ is a hyperbolic polynomial.
\edf

Hyperbolicity preserving operators have been studied by virtually every author who has studied hyperbolic polynomials (see \cite{CC04} and the references contained therein). The focus of our investigation involves the relationship between hyperbolicity preserving operators and hyperbolic polynomials with interlacing zeros.

\bdf \label{def:interlacing}
Let $f,g\in\R[x]$ with $\deg(f)=n$ and $\deg(g)=m$. We say that $f$ and $g$ have \emph{interlacing zeros}, if $f$ is hyperbolic with zeros $\alpha_1,\ldots,\alpha_n$ and $g$ is hyperbolic with zeros $\beta_1,\ldots,\beta_m$, where $|n-m|\le 1$, with one of the following forms holding:
\begin{enumerate}
\item $ \alpha_1 \le \beta_1 \le \alpha_2 \le \beta_2 \le \ldots \le \alpha_n \le \beta_m$, 
\item $ \beta_1 \le \alpha_1 \le \beta_2 \le \alpha_2 \le \ldots \le \beta_m \le \alpha_n$,
\item $ \alpha_1 \le \beta_1 \le \alpha_2 \le \beta_2 \le \ldots \le \beta_m \le \alpha_n$, or 
\item $ \beta_1 \le \alpha_1 \le \beta_2 \le \alpha_2 \le \ldots \le \alpha_n \le \beta_m$.
\end{enumerate}
We will also say that the zeros of any two hyperbolic polynomials of degree 0 or 1 interlace. 
\edf

\bdf \label{def:proper}
Given two non-zero polynomials $f,g\in\R[x]$, we say $f$ and $g$ are in \emph{proper position} and write $f\ll g$ if one of the following conditions holds:
\begin{enumerate}
\item $f$ and $g$ have interlacing zeros with form (1) or (4) in Definition \ref{def:interlacing} and the leading coefficients of $f$ and $g$ are of the same sign, or
\item $f$ and $g$ have interlacing zeros with form (2) or (3) in Definition \ref{def:interlacing}, and the leading coefficients of $f$ and $g$ are of opposite sign.
\end{enumerate}
We will say that the zero polynomial is in proper position with any other hyperbolic polynomial $f$ and write $0\ll f$ or $f\ll 0$. 
\edf

Notice that, by Definition \ref{def:proper}, if $f$ and $g$ are in proper position then $f$ and $g$ are hyperbolic. Also, to be clear, a non-zero constant can only be in proper position with another constant or a linear polynomial. However, the zero polynomial is in proper position with any hyperbolic polynomial. 

\bdf
For any two real polynomials $f$ and $g$, the \emph{Wronskian} of $f$ and $g$ is defined, on $\R$, by
\[
W[f,g]:=f(x)g'(x)-f'(x)g(x).
\]
\edf

It is a common exercise to show that for $f$ and $g$ with interlacing zeros, if $W[g,f]\le 0$ on the whole real line then $f\ll g$. 

The following Lemma demonstrates that proper position plays an important role in understanding hyperbolicity preservers. 

\begin{lem}[P. Br\"and\'en {\cite[Lemma 2.7]{B10}}] \label{thm:bra-int}
Suppose the linear operator 
\begin{equation} \label{eq:t-equ} 
T=\sum_{k=M}^N Q_k(x) D^k,
\end{equation}
where $Q_k(x)\in\R[x]$ for $M\le k\le N$, and $Q_M(x)Q_N(x)\not\equiv 0$, preserves hyperbolicity. Then $Q_j(x)\ll Q_{j+1}(x)$ for $M\le j\le N-1$. In particular, $Q_j(x)$ is hyperbolic or identically zero for all $M\le j\le N$. 
\end{lem}

In the special case for $M=0$ and $N=2$ in \eqref{eq:t-equ}, we find sufficient conditions that guarantee when $T$ preserves hyperbolicity. Our main result is the following: 

\begin{thm}  \label{thm:main} \label{thm:quad-wron'}
Suppose $Q_2,Q_1,Q_0$ are real polynomials such that $deg(Q_2)=2$, $deg(Q_1)\le 1$, $deg(Q_0)=0$. Then 
\begin{equation} \nonumber
T=Q_2D^2+Q_1D+Q_0
\end{equation}
preserves hyperbolicity if and only if
\begin{equation} \nonumber
W[Q_0,Q_2]^2-W[Q_0,Q_1]W[Q_1,Q_2]\le 0, \;\; \mbox{and} \;\; Q_0 \ll Q_1 \ll Q_2.
\end{equation}
\end{thm}

\section{Quadratic Hyperbolicity Preservers} 

We concern ourselves with operators of the following form.

\bdf We will call the second order differential operators of the form
\begin{equation} \label{eq:qua-op} 
T=Q(x)D^2+P(x)D+R(x) 
\end{equation}
a \emph{quadratic} operator, where the polynomials, $Q(x)$ is quadratic, $P(x)$ is linear, and $R(x)$ is constant. If \eqref{eq:qua-op} is also hyperbolicity preserving, then we will refer to it as a \emph{quadratic hyperbolicity preserver}. 
\edf

The following proposition presents an operator that has been quite influential to our exposition.

\begin{prop}[Forg\'acs et al. {\cite[Proposition 5]{bdfu}}]
If $0< d < 1$, then the operator 
\begin{equation}\nonumber
T=(x^2-1)D+2xD+d
\end{equation}
preserves hyperbolicity.
\end{prop} 

For motivation, we present several other similar examples of quadratic operators. 

\begin{exmp} \label{exmpops}
\begin{eqnarray} 
T_1 &=& (x^2-1)D^2 + 2xD - 1 \label{eq:t-1} \\ 
T_2 &=& (x^2-1)D^2 + 2xD + 0 \label{eq:t-2} \\ 
T_3 &=& (x^2-1)D^2 + 2xD + 1 \label{eq:t-3} \\ 
T_4 &=& (x^2-1)D^2 + 2xD + 2 \label{eq:t-4} \\
T_5 &=& (x^2-1)D^2 - 2xD - 1 \label{eq:t-5} \\
T_6 &=& (x^2-1)D^2 - 2xD + 0 \label{eq:t-6} \\
T_7 &=& (x^2-1)D^2 - 2xD + 1 \label{eq:t-7} \\
T_8 &=& (x^2-1)D^2 - 2xD + 2 \label{eq:t-8} 
\end{eqnarray} 
It was shown in \cite[Lemma 5]{bdfu} that \eqref{eq:t-3} is hyperbolicity preserving. Notice that $T_2=D(x^2-1)D$ and thus \eqref{eq:t-2} is hyperbolcity preserving as well. The other six examples can easily be shown to not preserve hyperbolicity. 
\begin{eqnarray} 
T_1[x^2-1] &=& 5x^2+2. \\ 
T_4[(x-10)^3] &=& 2(x-10)(7x^2-50x+97). \\
T_5[x^2] &=& -3x^2-2. \\
T_6[x^2] &=& -2x^2-2. \\
T_7[x^2] &=& -x^2-2. \\
T_8[(x-10)^3] &=& 2(x-10)(x^2+10x+97).
\end{eqnarray} 
\end{exmp}

These examples show that the property of interlacing coefficients is not sufficient to establish hyperbolicity preserving. Furthermore, \eqref{eq:t-4} demonstrates that the condition of proper position in Lemma \ref{thm:bra-int} is also not sufficient to establish hyperbolicity preserving. The examples motivate us to find the necessary and sufficient conditions on the polynomial coefficients in the quadratic operator \eqref{eq:qua-op}. 

We will completely characterize all quadratic hyperbolicity preservers. For our characterization, we will need a result due to J. Borcea and P. Br\"and\'en.

\begin{thm}[J. Borcea, P. Br\"and\'en {\cite[Theorem 1.3]{BB10a}}] \label{thm:bra-hyp1}
Let $T:\R[x]\to \R[x]$ be a finite differential linear operator, thus there exists real polynomials $\{Q_k(x)\}_{k=0}^n$ such that 
\[
T=\sum_{k=0}^n Q_k(x) D^k.
\] 
$T$ is hyperbolicity preserving, if and only if, 
\[
\sum_{k=0}^n Q_k(x)(-w)^k\not=0.
\] 
for every $x,w\in H^+$. 
\end{thm}

In general, Theorem \ref{thm:bra-hyp1} can be difficult to apply since very little is known about two variable \emph{stable} polynomials (See \cite{BB10a}). The next few lemmas establish a small class of two variable stable polynomials. 

\begin{lem} \label{lem:angle}
Let $A,\ B\in\C-\R$ be two non-real complex numbers such that
\emph{
\begin{enumerate}[(i)]
\item \label{itm:1} $0<\Arg(B)<\Arg(A)<2\pi$,
\item \label{itm:2} $\Arg(A)-\Arg(B)<\pi$, and
\item \label{itm:3} $\Im(A)<\Im(B)$.
\end{enumerate}}
\noindent Then for any $r_1,\ r_2 \in\R$, $r_1<r_2$, there is $x,w\in H^+$ such that $(x+r_1)w=A$ and $(x+r_2)w=B$. 
\end{lem}

\begin{proof} Consider the following cases.

\smallskip

\noindent Case 1: $B\in H^+$. The point $B$ may be located in either quadrant I, on the imaginary axis, or in quadrant II, as described in Figure \ref{eq:B+regions}. The hypotheses \eqref{itm:1}, \eqref{itm:2}, and \eqref{itm:3} implies that point $A$ is located somewhere in the shaded region of the corresponding point $B$.
\begin{figure}[h!]
\begin{tikzpicture}[scale=1.7][domain=0:4]
\fill[color=gray!50!white] (0,0) -- (.6,.8) -- (-1,.8) -- (-1,-1) -- (-.75,-1) -- cycle;
\fill (.6,.8) circle (0.8 pt) node[above]{$B$};
\draw[thick, dotted] (.6,.8) -- (-1,.8);
\draw[thick, dotted] (0,0) -- (.6,.8);
\draw[thick, dotted] (0,0) -- (-.75,-1);
\draw[<->] (-1,0) -- (0,0) -- (1,0);
\draw[<->] (0,-1) -- (0,0) -- (0,1);
\end{tikzpicture} 
\qquad 
\begin{tikzpicture}[scale=1.7][domain=0:4]
\fill[color=gray!50!white] (0,0) -- (0,.8) -- (-1,.8) -- (-1,-1) -- (0,-1) -- cycle;
\fill (0,.8) circle (0.8 pt) node[right]{$B$};
\draw[thick, dotted] (0,.8) -- (-1,.8);
\draw[thick, dotted] (0,0) -- (0,.8);
\draw[thick, dotted] (0,0) -- (0,-1);
\draw[<->] (-1,0) -- (0,0) -- (1,0);
\draw[<->] (0,-1) -- (0,0) -- (0,1);
\end{tikzpicture}
\qquad
\begin{tikzpicture}[scale=1.7][domain=0:4]
\fill[color=gray!50!white] (0,0) -- (-.6,.8) -- (-1,.8) -- (-1,-1) -- (.75,-1) -- cycle;
\fill (-.6,.8) circle (0.8 pt) node[above]{$B$};
\draw[thick, dotted] (-.6,.8) -- (-1,.8);
\draw[thick, dotted] (0,0) -- (-.6,.8);
\draw[thick, dotted] (0,0) -- (.75,-1);
\draw[<->] (-1,0) -- (0,0) -- (1,0);
\draw[<->] (0,-1) -- (0,0) -- (0,1);
\end{tikzpicture}
\caption{}
\label{eq:B+regions}
\end{figure}
Define the function $f:[0,\Arg(B)]\to \R$ by
\beq
f(\theta):=\Im(e^{-i\theta}A)-\Im(e^{-i\theta}B).
\eeq
Then $f(0)<0$ by \eqref{itm:3}, and $f(\Arg(B))>0$ by \eqref{itm:2}. Thus by continuity, there exist $\theta_0\in(0,\Arg(B))$ such that $f(\theta_0)=0$, which implies that $(e^{-i\theta_0}B-e^{-i\theta_0}A)>0$ by \eqref{itm:1}. Define the function $g:[0,\infty)\to \R$ by
\beq
g(k):=k(e^{-i\theta_0}B-e^{-i\theta_0}A).
\eeq
Notice $g\ge 0$, $g(0)=0$, and $\disp \lim_{k\to +\infty} g(k)=+\infty$. Thus, there exist $k_0>0$ such that $g(k_0)=r_2-r_1$. Let
\beq
x=\frac{1}{2}(k_0 e^{-i\theta_0}B+k_0 e^{-i\theta_0}A -r_1-r_2), \mbox{ and } w=\frac{1}{k_0}e^{i\theta_0}.
\eeq
It follows that $x,w\in H^+$, $(x+r_1)w=A$, and $(x+r_2)w=B$.

\smallskip

\noindent Case 2: $B\in H^-$. Similar to Case 1, the point $B$ may be located in either quadrant III, on the imaginary axis, or in quadrant IV, as described in Figure \eqref{eq:B-regions}. Point $A$ is located somewhere in the shaded region of the corresponding point $B$ by hypotheses \eqref{itm:1}, \eqref{itm:2}, and \eqref{itm:3}.
\begin{figure}[h!]
\begin{tikzpicture}[scale=1.7][domain=0:4]
\fill[color=gray!50!white] (1,-.4) -- (-.6,-.4) -- (-1,-2/3) -- (-1,-1.5) -- (1,-1.5) -- cycle;
\fill (-.59,-.4) circle (0.8 pt) node[above]{$B$};
\draw[thick, dotted] (-.6,-.4) -- (1,-.4);
\draw[thick, dotted] (-1,-2/3) -- (-.6,-.4);
\draw[thick, dotted] (0,0) -- (-.6,-.4);
\draw[<->] (-1,0) -- (0,0) -- (1,0);
\draw[<->] (0,-1.5) -- (0,0) -- (0,.5);
\end{tikzpicture} \qquad 
\begin{tikzpicture}[scale=1.7][domain=0:4]
\fill[color=gray!50!white] (1.5,-.4) -- (0,-.4) -- (0,-1.5) -- (1.5,-1.5) -- cycle;
\fill (0,-.4) circle (0.8 pt) node[left]{$B$};
\draw[thick, dotted] (0,0) -- (0,-1.5);
\draw[thick, dotted] (0,-.4) -- (1.5,-.4);
\draw[<->] (-.5,0) -- (0,0) -- (1.5,0);
\draw[<->] (0,-1.5) -- (0,0) -- (0,.5);
\end{tikzpicture}
\qquad \begin{tikzpicture}[scale=1.7][domain=0:4]
\fill[color=gray!50!white] (1.5,-.4) -- (0.2,-.4) -- (.75,-1.5) -- (1.5,-1.5) -- cycle;
\fill (.2,-.4) circle (0.8 pt) node[left]{$B$};
\draw[thick, dotted] (0,0) -- (.75,-1.5);
\draw[thick, dotted] (.2,-.4) -- (1.5,-.4);
\draw[<->] (-.5,0) -- (0,0) -- (1.5,0);
\draw[<->] (0,-1.5) -- (0,0) -- (0,.5);
\end{tikzpicture}
\caption{}
\label{eq:B-regions}
\end{figure}

Define the function $f:[0,2\pi-\Arg(B)]\to \R$ by
\beq
f(\theta):=\Im(e^{i\theta}A)-\Im(e^{i\theta}B).
\eeq
Then $f(0)<0$ by \eqref{itm:3}, and $f(2\pi-\Arg(B))>0$ by \eqref{itm:2}. Thus by continuity, there exist $\theta_0\in(0,2\pi-\Arg(B))$ such that $f(\theta_0)=0$, which implies that $(e^{i\theta_0}B-e^{i\theta_0}A)<0$ by \eqref{itm:1}. Define the function $g:(-\infty,0]\to \R$ by
\beq
g(k):=k(e^{i\theta_0}B-e^{i\theta_0}A).
\eeq
Then $g\ge 0$, $g(0)=0$, and $\disp \lim_{k\to -\infty} g(k)=+\infty$. Thus, there exist $k_0<0$ such that $g(k_0)=r_2-r_1$. Let
\beq
x=\frac{1}{2}(k_0 e^{i\theta_0}B+k_0 e^{i\theta_0}A -r_1-r_2), \mbox{ and } w=\frac{1}{k_0}e^{-i\theta_0}.
\eeq
It follows that $x,w\in H^+$, $(x+r_1)w=A$, and $(x+r_2)w=B$.
\end{proof}

\begin{lem} \label{lem:neg}
Let $a,b,r_1,r_2,r \in\R$, $a,b\ge 0$, and $r_1 \neq r_2$. Set
\beq \label{eq:neqr}
f(x,w)=((x+r_1)w-a)((x+r_2)w-b), \quad x,w\in\C.
\eeq
Then 
\[
f(x,w)\neq r\ \ \forall\ x,w\in H^+ \quad \mbox{if and only if}\quad r\in [0,ab].
\] 
\end{lem}

\begin{proof} Since the factors of $f(x,w)$ in \eqref{eq:neqr} are symmetric, we let $r_1<r_2$. There are three cases to prove necessity. The following is the outline.

Case 1. $r\in (-\infty,0)$, and $a<b+2\sqrt{|r|}$. 

Case 2. $r\in (-\infty,0)$, and $a\ge b+2\sqrt{|r|}$. 

Case 3. $r\in (ab,\infty)$. 

\noindent We show in each case that there exist $x,w\in H^+$ such that $f(x,w)=r$.

\smallskip

\noindent Case 1. Consider $r\in (-\infty,0)$, and $a<b+2\sqrt{|r|}$. Define $g:[0,\pi/2]\to\R$ by
\beq
\barr{rl}
g(\theta):= & \!\!\! \left(\sqrt{|r|} e^{i \theta}+b\right)- \left( \sqrt{|r|}e^{i(\pi-\theta)}+a\right)\\
= & \!\!\! \sqrt{|r|}(2\cos(\theta))-a+b.
\earr
\eeq
The function $g$ is real valued and $g(0)=b+2\sqrt{|r|}-a>0$ by assumption. Thus by continuity, there exists $\theta_0\in (0,\pi/2)$ such that $g(\theta_0)>0$, which implies the following.

\begin{enumerate} [\qquad (a)]
\item \label{itm:1'} $\Im\left(\sqrt{|r|} e^{i\theta_0}+b\right)- \Im\left( \sqrt{|r|}e^{i(\pi-\theta_0)}+a\right)=0$,
\item \label{itm:2'} $\Re\left(\sqrt{|r|} e^{i\theta_0}+b\right)- \Re\left( \sqrt{|r|}e^{i(\pi-\theta_0)}+a\right)>0$, and
\item \label{itm:3'} $\left(\sqrt{|r|} e^{i\theta_0}+b\right),\ \left( \sqrt{|r|}e^{i(\pi-\theta_0)}+a\right)\in H^+$.
\end{enumerate}
By \eqref{itm:1'}, \eqref{itm:2'}, and \eqref{itm:3'},
\begin{equation}
\Arg\left( \sqrt{|r|}e^{i(\pi-\theta_0)}+a\right)-\Arg\left(\sqrt{|r|} e^{i\theta_0}+b\right)>0.
\end{equation}
Define the function $h:(0,1]\to\R$ by
\begin{equation}
h(k):=\Arg\left(k\sqrt{|r|}e^{i(\pi-\theta_0)}+a\right)-\Arg\left(\frac{\sqrt{|r|}}{k} e^{i \theta_0}+b\right).
\end{equation}
The function $h$ is real valued, and $h(1)>0$. Thus by continuity, there exists $k_0<1$ such that
\begin{equation} \label{eq:c1i}
\Arg\left(k_0\sqrt{|r|}e^{i(\pi-\theta_0)}+a\right)-\Arg\left(\frac{\sqrt{|r|}}{k_0} e^{i\theta_0}+b\right)>0,
\end{equation}
such that 
\beq \label{eq:c1ii}
\Im\left(k_0\sqrt{|r|}e^{i(\pi-\theta_0)}+a\right)<\Im\left(\frac{\sqrt{|r|}}{k_0} e^{i\theta_0}+b\right).
\eeq
Let
\begin{equation}
A=k_0\sqrt{|r|}e^{i(\pi-\theta_0)}+a, \quad \mbox{ and } \quad B=\frac{\sqrt{|r|}}{k_0} e^{i\theta_0}+b.
\end{equation}
Then \eqref{eq:c1i} and \eqref{eq:c1ii} satisfies items \eqref{itm:1}, \eqref{itm:2}, and \eqref{itm:3} of Lemma \ref{lem:angle}, hence there exist $x,w\in H^+$ such that $(x+r_1)w=A$ and $(x+r_2)w=B$. Thus,
\begin{equation}
\barr{rl}
f(x,w) \!\!\! & =((x+r_1)w-a)((x+r_2)w-b)\\
& = \disp \left(k_0\sqrt{|r|}e^{i(\pi-\theta_0)}\right)\left(\frac{\sqrt{|r|}}{k_0} e^{i\theta_0}\right) = -|r| = r.
\earr
\end{equation}

\noindent Case 2: We consider $r\in (-\infty,0)$, and $b+2\sqrt{|r|}\le a$. We will only need $b<a+2\sqrt{|r|}$. This is easily seen to be true by adding $2\sqrt{|r|}$ to both sides of $b+2\sqrt{|r|}\le a$, and observing $b< b+4\sqrt{|r|}$. Define the function $g:[0,\pi/2]\to\R$ by
\begin{equation} 
\barr{rl}
g(\theta)  := & \!\!\! \left( \sqrt{|r|}e^{i(2\pi-\theta)}+a\right) - \left(\sqrt{|r|} e^{i(\pi+\theta)}+b\right)\\
 =  & \!\!\! \sqrt{|r|}(2\cos(\theta))+a-b. 
\earr
\end{equation}
Again, $g$ is real valued, and $g(0)=a+2\sqrt{|r|}-b>0$. Thus by continuity, there exists $\theta_0\in (0,\pi/2)$ such that $g(\theta_0)>0$, which implies the following: 
\begin{enumerate}[\qquad (a)]
\item \label{itm:1''} $\Im\left( \sqrt{|r|}e^{i(2\pi-\theta_0)}+a\right) - \Im\left(\sqrt{|r|} e^{i(\pi+\theta_0)}+b\right) =0$,
\item \label{itm:2''} $\Re\left( \sqrt{|r|}e^{i(2\pi-\theta_0)}+a\right) - \Re\left(\sqrt{|r|} e^{i(\pi+\theta_0)}+b\right)>0$,
\item \label{itm:3''} $\left( \sqrt{|r|}e^{i(2\pi-\theta_0)}+a\right),\ \left(\sqrt{|r|} e^{i(\pi+\theta_0)}+b\right)\in H^-$.
\end{enumerate} 
By \eqref{itm:1''}, \eqref{itm:2''}, and \eqref{itm:3''},
\begin{equation}
\Arg\left( \sqrt{|r|}e^{i(2\pi-\theta_0)}+a\right)-\Arg\left(\sqrt{|r|} e^{i(\pi+\theta_0)}+b\right)>0. 
\end{equation}
Define the function $h:[1,\infty)\to \R$ by
\begin{equation}
h(k):=\Arg\left(k\sqrt{|r|}e^{i(2\pi-\theta_0)}+a\right)-\Arg\left(\frac{\sqrt{|r|}}{k} e^{i(\pi+\theta_0)}+b\right).
\end{equation}
The function $h$ is real valued, and $h(1)>0$. Thus by continuity, there exists $k_0>1$ such that
\begin{equation} \label{eq:c2i}
\Arg\left(k_0\sqrt{|r|}e^{i(2\pi-\theta_0)}+a\right)-\Arg\left(\frac{\sqrt{|r|}}{k_0} e^{i(\pi+\theta_0)}+b\right)>0,
\end{equation}
so that
\beq \label{eq:c2ii}
\Im\left(k_0\sqrt{|r|}e^{i(2\pi-\theta_0)}+a\right)<\Im\left(\frac{\sqrt{|r|}}{k_0} e^{i(\pi+\theta_0)}+b\right).
\eeq
Let
\begin{equation}
A=k_0\sqrt{|r|}e^{i(2\pi-\theta_0)}+a, \quad \mbox{ and } \quad B=\frac{\sqrt{|r|}}{k_0} e^{i(\pi+\theta_0)}+b.
\end{equation}
Then \eqref{eq:c2i} and \eqref{eq:c2ii} satisfies items \eqref{itm:1}, \eqref{itm:2}, and \eqref{itm:3} of Lemma \ref{lem:angle}, hence there exist $x,w\in H^+$ such that $(x+r_1)w=A$ and $(x+r_2)w=B$. Thus,
\begin{equation}
\barr{rl}
f(x,w) \!\!\! & =((x+r_1)w-a)((x+r_2)w-b) \\
& = \disp \left(k_0\sqrt{|r|} e^{i(2\pi-\theta_0)}\right) \left(\frac{\sqrt{|r|}}{k_0} e^{i(\pi+\theta_0)}\right) = -|r| = r.
\earr
\end{equation}

\noindent Case 3: We consider $r\in (ab,\infty)$. Since $r>ab$, $r=a'b'$, for some $a'>a$, and $b'>b$. Define the function $g:[\pi/2,\pi]\to [a-a',a]\times [b-b',b]$ by
\begin{equation}
g(\theta):=\left( \Re(a'e^{-i\theta})+a , \,\Re(b'e^{i\theta})+b \right).
\end{equation}
Since $a-a',b-b'<0$, $\,g(\pi)=(a-a',b-b')$ has negative coordinates. By continuity, there exists $\theta_0\in (\pi/2,\pi)$ such that $g(\theta_0)$ has negative coordinates, which implies that $a'e^{-i\theta_0}+a$ is in quadrant three, and $b'e^{i\theta_0}+b$ is in quadrant two. Let
\begin{equation}
A=a' e^{-i\theta_0}+a, \quad \text{ and } \quad B=b' e^{i\theta_0}+b.
\end{equation}
Again, by Lemma \ref{lem:angle}, there exist $x,w\in H^+$ such that $(x+r_1)w=A$, and $(x+r_2)w=B$. Thus, 
\begin{equation}
f(x,w)=((x+r_1)w-a)((x+r_2)w-b)=\left(a' e^{-\theta_0 i}\right) \left(b' e^{\theta_0 i}\right) = a'b'=r. 
\end{equation}

To prove sufficiency, first consider $r\in (0,ab]$. By way of contradiction, assume there exist $x,w\in H^+$ such that $((x+r_1)w-a)((x+r_2)w-b)=r$. Let $A=((x+r_1)w-a)$, $B=((x+r_2)w-b)$. Since $x+r_1,x+r_2\in H^+$, the rotation by $\Arg(w)\in (0,\pi)$ and the shifts to the left by $a,b>0$ restrict the location of $A$ and $B$ considerably. Indeed, since $AB$ is a positive real number, $\Arg(A)+\Arg(B)=0$ (mod $2\pi$). In particular, as $r_1<r_2$, $B$ must be in $H^+$, which implies
\beq\label{ine1}
0<\Arg(w)<\Arg((x+r_2)w) < \Arg((x+r_2)w-b) < \pi,
\eeq
and $A$ must be in $H^-$, which implies
\beq\label{ine2}
\pi < \Arg((x+r_1)w-a) < \Arg((x+r_1)w)<\pi-\Arg(w)<2\pi.
\eeq
The following figure illustrates inequalities \eqref{ine1} and \eqref{ine2}.
\begin{equation} \nonumber \begin{tikzpicture}[scale=4]
\draw[<->] (-1.7,0) -- (0,0) node[below right] {Origin} -- (.7,0) node[above] {Real Line} -- (.9,0);
\draw[<->] (240:.8) -- (60:1);
\draw[-] (0:.3) arc(0:30:.3) node[right] {$\Arg(w)$};
\draw[->] (30:.3) arc(30:60:.3);
\fill (0,0) circle (0.4 pt);
\draw[dotted] (60:.8) -- (.2500,.6928) node[below] {$\epsilon$} -- (.1000,.6928) node[above] {$(x+r_2)w$};
\draw[-] (.1000,.6928) -- (-.4000,.6928) node[below] {$b$} -- (-1.5,.6928) node[above] {$B=(x+r_2)w-b$};
\draw[dotted] (0,0) -- (.1000,.6928); \draw[-] (0,0) -- (-1.5,.6928);
\fill (60:.8) circle (0.4 pt); \fill (.1,.6928) circle (0.4 pt); \fill (-1.5,.6928) circle (0.4 pt);
\draw[-] (-155.21:1.074) node[below left] {$A=(x+r_1)w-a$} -- (-.75,-.4502) node[above] {$a$} -- (-.5599,-.4502);
\draw[dotted] (-.5599,-.4502) node[below] {$(x+r_1)w$} -- (-.38,-.4502) node[above] {$\delta$} -- (-.2599,-.4502);
\draw[dotted] (0,0) -- (-.5599,-.4502); \draw[-] (0,0) -- (-155.21:1.074);
\fill (-155.21:1.074) circle (0.4 pt); \fill (-.5599,-.4502) circle (0.4 pt); \fill (-.2599,-.4502) circle (0.4 pt);
\end{tikzpicture}\end{equation}
We let $\epsilon$ and $\delta$ be the horizontal distance from $(x+r_1)w$ and $(x+r_2)w$ to the line formed by $\Arg(w)$. In fact, $\delta=\frac{\Im(x+r_1)}{\sin(Arg(w))}$, and $\epsilon=\frac{\Im(x+r_2)}{\sin(Arg(w))}$, so that $\delta=\epsilon>0$.  We redraw the picture with different labels and examine the points geometrically.  
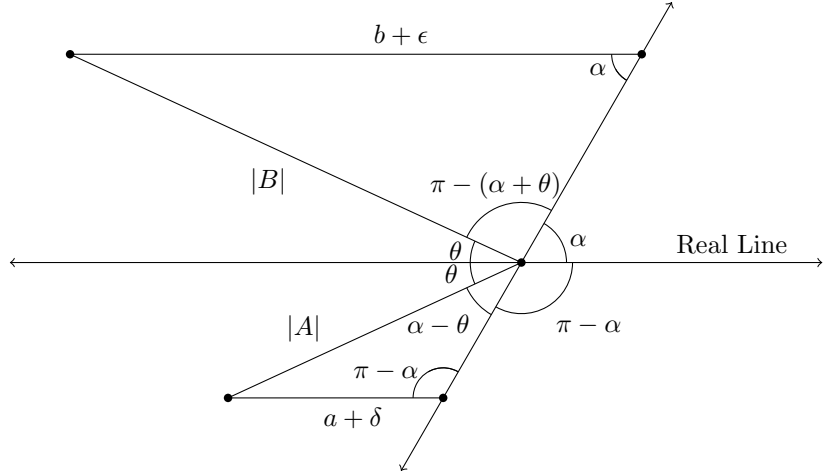
\begin{equation} \nonumber \begin{tikzpicture}[scale=4]
\draw[<->] (-1.7,0) -- (0,0) -- (.7,0) node[above] {Real Line} -- (1,0);
\draw[<->] (240:.8) -- (60:1);
\fill (0,0) circle (0.4 pt);
\draw[-] (0:.15) arc(0:30:.15) node[right] {$\alpha$}; \draw[-] (30:.15) arc(30:60:.15);
\draw[-] (60:.2) arc(60:115:.2) node[above] {$\pi-(\alpha+\theta)$}; \draw[-] (115:.2) arc(115:155.21:.2);
\draw[-] (155.21:.17) arc(155.21:167:.17) node[left] {$\theta$}; \draw[-] (167:.17) arc(167:180:.17);
\draw[-] (-180:.17) arc(-180:-167:.17) node[left] {$\theta$\,}; \draw[-] (-167:.17) arc(-167:-155.21:.17);
\draw[-] (-155.21:.2) arc(-155.21:-135:.2) node[below left] {$\alpha-\theta$}; \draw[-] (-135:.2) arc(-135:-120:.2);
\draw[-] (-120:.17) arc(-120:-60:.17) node[below right] {$\pi-\alpha$}; \draw[-] (-60:.17) arc(-60:0:.17);
\draw[-] +(60:.7) arc(240:210:.1) node[left] {$\alpha$}; \draw[-] +(60:.7) arc(240:180:.1);
\draw[-] +(-120:.4198) arc(60:120:.1) node[left] {$\pi-\alpha$}; \draw[-] +(-120:.4198) arc(60:180:.1);
\draw[-] (60:.8) -- (-.4000,.6928) node[above] {$b+\epsilon$} -- (-1.5,.6928);
\draw[-] (0,0) -- (-.75,.3464) node[below left] {$|B|$} -- (-1.5,.6928);
\fill (60:.8) circle (0.4 pt); \fill (-1.5,.6928) circle (0.4 pt);
\draw[-] (-155.21:1.074) -- (-.5599,-.4502) node[below] {$a+\delta$} -- (-.2599,-.4502);
\draw[-] (0,0) -- (-155.21:.7) node[above left] {$|A|$} -- (-155.21:1.074);
\fill (-155.21:1.074) circle (0.4 pt); \fill (-.2599,-.4502) circle (0.4 pt);
\end{tikzpicture}\end{equation}
The inequalities $\alpha-\theta>0$ and $\pi-(\alpha+\theta)>0$ imply
$0<\theta<\alpha<\pi-\theta<\pi$, so that
\[
\sin(\theta)< \sin(\alpha),
\]
since $\sin(\theta)=\sin(\pi-\theta)$. Thus,
\beq
0<\left( \frac{\sin(\theta)}{\sin(\alpha)}\right)^2< 1,
\eeq
and the law of sines yield that
\begin{equation}
\barr{rl}
(a+\delta)(b+\epsilon) \! \! & \disp = \! \;\; \frac{|A|\sin(\alpha-\theta)}{\sin(\pi-\alpha)} \cdot \frac{|B|\sin(\pi-(\alpha+\theta))}{\sin(\alpha)}\\ 
& =\! \;\; \disp \left(1\! -\! \left(\frac{\sin(\theta)}{\sin(\alpha)} \right)^2\right)\!|AB|<|AB|.
\earr
\end{equation}
Hence we have the contradiction that
\begin{equation}
ab<(a+\delta)(b+\epsilon)< |AB|=r.
\end{equation} 
To finish the proof, consider $r=0$. By way of contradiction, suppose there are $x,w\in H^+$ such that
\[
((x+r_1)w-a)((x+r_2)w-b)=0.
\]
Thus, $(x+r_1)w=a$, or $(x+r_2)w=b$. However, neither of these can hold, since the product of any two complex numbers in $H^+$ cannot be a non negative real number. 
\end{proof}

\begin{thm} \label{thm:quadop}
Let $a,b\ge 0$, $r_1,r_2,R\in\R$, and $r_1\not=r_2$. Then, 
\begin{equation} \nonumber
R\in[0,ab]
\end{equation}
if and only if 
\begin{equation} \nonumber
T:=(x+r_1)(x+r_2)D^2+\left(b(x+r_1)+a(x+r_2)\right)D+R
\end{equation} 
(where $D:=\frac{d}{dx}$), is hyperbolicity preserving.
\end{thm}

\begin{proof}
($\Rightarrow$) Assume $R\in [0,ab]$. By Theorem \ref{thm:bra-hyp1}, it suffices to show for every $x,w\in H^+$,
\begin{equation} \label{eq:qua-hyp1}
(x+r_1)(x+r_2)w^2-\left(b(x+r_1)+a(x+r_2)\right)w+R\not=0.
\end{equation}
We assume on the contrary that \eqref{eq:qua-hyp1} is false for some $x,w\in H^+$. We factor \eqref{eq:qua-hyp1} to attain
\begin{equation}
((x+r_1)w-a)((x+r_2)w-b)=ab-R,
\end{equation}
which is impossible by Lemma \ref{lem:neg}, a contradiction. 

\smallskip

($\Leftarrow$) Suppose $T$ is hyperbolicity preserving. By Theorem \ref{thm:bra-hyp1}, for every $x,w\in H^+$,
\beq \label{eq:qua-hyp1'}
(x+r_1)(x+r_2)w^2-\left(b(x+r_1)+a(x+r_2)\right)w+R\not=0.
\eeq
We factor \eqref{eq:qua-hyp1'} to attain
\beq
((x+r_1)w-a)((x+r_2)w-b)\not=ab-R,\ \ \ \forall\ x,w\in H^+,
\eeq
which implies that $R\in [0,ab]$ by Lemma \ref{lem:neg}. 
\end{proof}

\begin{thm}\label{thm:quad-op}
For $c_i,r_j\in\R$, $i=0,1,2$, $j=1,2,3$, $c_2\neq 0$, $r_1\neq r_2$, let $Q_0(x)=c_0$, $Q_1(x)=c_1(x-r_3)$, $Q_2(x)=c_2(x-r_1)(x-r_2)$. Then 
\[
0 \le c_1^2 \left (\frac{(r_1-r_3)(r_3-r_2)}{(r_2-r_1)^2}\right)-c_0c_2,
\]
and $c_0,c_1,c_2$ are of the same sign if and only if
\[
T:=Q_2(x)D^2+Q_1(x)D+Q_0(x)
\]
(where $ D:=\frac{d}{dx}$), preserves hyperbolicity. 
\end{thm}

\begin{proof}
To prove sufficiency, if $T$ preserves hyperbolicity, then by Lemma \ref{thm:bra-int}, $c_i$, $i=0,1,2$ are of the same sign and the zeros of $Q_2$ and $Q_1$ interlace. Since
{\allowdisplaybreaks
\begin{align}
T = & \; c_2 \lr{\!(x-r_1)(x-r_2)D^2+ \frac{c_1}{c_2}(x-r_3) D+ \frac{c_0}{c_2}} \label{intlac1} \\
 = & \; c_2 \biggl((x-r_1)(x-r_2)D^2 \nonumber \\
& \qquad + \frac{c_1}{c_2}\!\br{\!\frac{(r_1-r_3)}{(r_1-r_2)} (x-r_2)+\frac{(r_3-r_2)}{(r_1-r_2)}(x-r_1)\!}\!D+ \frac{c_0}{c_2}\biggr), \label{intlac2}
\end{align}}
then by Theorem \ref{thm:quadop}, 
\beq
\frac{c_0}{c_2}\in\br{0, \; \lr{\frac{c_1}{c_2}}^2\frac{(r_1-r_3)(r_3-r_2)}{(r_1-r_2)^2}},
\eeq
and so
\beq
0 \le c_1^2 \left (\frac{(r_1-r_3)(r_3-r_2)}{(r_2-r_1)^2}\right)-c_0c_2.
\eeq
To prove necessity, suppose $c_i$, $i=0,1,2$ are of the same sign, and 
\beq \label{qosuff}
0 \le c_1^2 \left (\frac{(r_1-r_3)(r_3-r_2)}{(r_2-r_1)^2}\right)-c_0c_2.
\eeq
We want to conclude that
\beq \label{qosuffpos}
\frac{c_1}{c_2}\frac{(r_1-r_3)}{(r_1-r_2)}, \frac{c_1}{c_2}\frac{(r_3-r_2)}{(r_1-r_2)}\ge 0.
\eeq
To this end, if $c_1=0$, then \eqref{qosuffpos} holds immediately. Suppose $c_1\neq 0$, and that $r_1<r_2$. Then \eqref{qosuff} implies $0\le(r_1-r_3)(r_3-r_2)$, and we conclude that $r_1\le r_3\le r_2$ (i.e., $r_3<r_1<r_2$ cannot hold, since it implies $(r_1-r_3)(r_3-r_2)<0$, and also if $r_1<r_2<r_3$, then $(r_1-r_3)(r_3-r_2)<0$), and hence, \eqref{qosuffpos} holds. By symmetry, the same conclusion is true if $r_2<r_1$. Thus by Theorem \ref{thm:quadop}, $T$ preserves hyperbolicity.
\end{proof}

The equality of \eqref{intlac1} and \eqref{intlac2} uses a fact established in Fisk's polynomial book \cite[p.\; 13, Lemma 1.20]{fisk08}, although because our case is easy to verify, we do not need its full strength. For the sake of completeness, we state the result seen in Fisk's book.

\begin{lem} [Fisk {\cite[p.\; 13, Lemma 1.20]{fisk08}}]
Assume that $f$ is a polynomial of degree $n$, with positive leading coefficient, and with real zeros $\set{a_1,\ldots, a_n}$. Suppose that $g$ is a polynomial with positive leading coefficient. If $g$ has degree $n-1$, and we write
\[
g(x)=c_1 \frac{f(x)}{x-a_1}+\ldots+c_n \frac{f(x)}{x-a_n},
\]
then $f$ and $g$ have interlacing zeros if and only if all $c_i\ge 0$ for $i=1,2,\ldots,n$. 
\end{lem}

We now remove the condition of $Q_2$ having distinct zeros. We begin with a lemma that is analogous to Lemma \ref{lem:neg}.

\begin{lem} \label{lem:neg'}
Let $a,r \in\R$, $a\ge 0$. Set 
\[
f(x):=x^2-ax+r, \quad x\in\C.
\]
Then 
\[
f(x)\neq 0\ \ \forall \; x\in\C - [0,\infty) \quad \mbox{if and only if} \quad r\in \br{0,\frac{a^2}{4}}.
\] 
\end{lem}

\begin{proof}
The zeros of $f$ are $\frac{1}{2}\lr{a\pm\sqrt{a^2-4r}}$. There are two cases to prove necessity. 

Case 1. If $r<0$, then one of the zeros of $f$ is a negative real number, thus there exist $x\in\C - [0,\infty)$ such that $f(x)=0$.

Case 2. If $r>a^2/4$, then $f$ has two imaginary zeros, thus the zeros of $f$ are in $\C - [0,\infty)$.

To prove sufficiency, suppose $0\le r \le a^2/4$. Then $f$ has two non-negative zeros, so that $f$ never vanishes in $\C - [0,\infty)$.
\end{proof}

\begin{thm} \label{thm:mltrtquadop}
Let $a\ge 0$, $r,R\in\R$. Then, 
\begin{equation} \nonumber
R\in \br{0,\frac{a^2}{4}}
\end{equation}
if and only if 
\begin{equation} \nonumber
T:=(x+r)^2 D^2+ a(x+r)D+R
\end{equation} 
is hyperbolicity preserving.
\end{thm}

\begin{proof}
($\Rightarrow$) Assume $R\in [0,a^2/4]$. By Theorem \ref{thm:bra-hyp1}, it suffices to show for every $x,w\in H^+$,
\begin{equation} \label{eq:qua-hyp1''}
(x+r)^2w^2- a(x+r)w+R \not= 0.
\end{equation}
We assume on the contrary that \eqref{eq:qua-hyp1''} is false for some $x,w\in H^+$. Let $z=(x+r)w$ in \eqref{eq:qua-hyp1''}, so that $z\in \C -[0,\infty)$, and
\begin{equation}
z^2-a z +R=0.
\end{equation}
This is impossible by Lemma \ref{lem:neg'}, a contradiction. 

\smallskip

\noindent ($\Leftarrow$) Suppose $T$ is hyperbolicity preserving. By Theorem \ref{thm:bra-hyp1}, for every $x,w\in H^+$,
\begin{equation} \label{eq:qua-hyp1''''}
(x+r)^2 w^2-a(x+r)w+R\not=0.
\end{equation}
Let $z=(x+r)w$ in \eqref{eq:qua-hyp1''''}, so that $z\in \C -[0,\infty)$, and
\begin{equation}
z^2-az + R\not=0,\ \ \ \forall\ z\in \C-[0,\infty),
\end{equation}
which implies that $R\in [0,a^2/4]$ by Lemma \ref{lem:neg'}. 
\end{proof}

The analogous statement of Theorem \ref{thm:quad-op} is the following, and its proof follows \emph{mutatis mutandis}, from the proof of Theorem \ref{thm:quad-op}.

\begin{thm}\label{thm:quad-op'}
For $r,c_i\in\R$, $i=0,1,2$, $c_2\neq 0$, let $Q_0(x)=c_0$, $Q_1(x)=c_1(x-r)$, $Q_2(x)=c_2(x-r)^2$. Then  
\begin{equation} \nonumber
0 \le c_1^2 \left( \frac{1}{4} \right) - c_0c_2
\end{equation}
and $c_0,c_1,c_2$ are of the same sign, if and only if 
\begin{equation} \nonumber
T=Q_2(x)D^2+Q_1(x)D+Q_0(x)
\end{equation}
preserves hyperbolicity. 
\end{thm}

We now wish to find a condition that combines the statements of Theorem \ref{thm:quad-op'} and Theorem \ref{thm:quad-op}. But first a Lemma. 

\begin{lem} \label{thm:quarter}
For $c_i,r_j\in\R$, $i=0,1,2$, $j=1,2,3$, $c_2\neq 0$, let $Q_0(x)=c_0$, $Q_1(x)=c_1(x-r_3)$, $Q_2(x)=c_2(x-r_1)(x-r_2)$. If $T=Q_2(x)D^2+Q_1(x)D+Q_0(x)$ is hyperbolicity preserving then
\[
0\le c_1^2 -  4 c_0 c_2.
\]
Furthermore, if $r_1\not=r_2$ then 
\[
0\le c_1^2 \frac{(r_1-r_3)(r_3-r_2)}{(r_2-r_1)^2} - c_0 c_2\le c_1^2\frac{1}{4} - c_0 c_2. 
\]
Thus, if $0=c_1^2-4 c_0 c_2$ then $2r_3=r_1+r_2$. 
\end{lem}

\begin{proof}
Theorem \ref{thm:quad-op'} deals with the case of when $r_1=r_2$, thus it suffices to show
\beq
0\le \frac{(r_1-r_3)(r_3-r_2)}{(r_2-r_1)^2}\le \frac{1}{4}.
\eeq
The left inequality holds because $Q_2$ and $Q_1$ have interlacing zeros by Lemma \ref{thm:bra-int}. To show the right inequality we proceed as follows, 
\beq
0\le (2r_3-(r_1+r_2))^2,
\eeq
\beq
4(r_1r_3+r_2r_3) \le (r_2+r_1)^2 +4 r_3^2,
\eeq
\beq
4(r_1r_3-r_1r_2-r_3^2+r_2r_3) \le r_2^2-2r_1r_2+r_1^2,
\eeq
\beq
4(r_1-r_3)(r_3-r_2) \le (r_2-r_1)^2. \qedhere 
\eeq
\end{proof}

\begin{thm} \label{thm:quad-wron}
For $c_i,r_j\in\R$, $i=0,1,2$, $j=1,2,3$, $c_2\neq 0$, let $Q_0(x)=c_0$, $Q_1(x)=c_1(x-r_3)$, $Q_2(x)=c_2(x-r_1)(x-r_2)$ with $Q_0(x)\ll Q_1(x)\ll Q_2(x)$. Then 
\[
T=Q_2(x) D^2+ Q_1(x) D + Q_0(x)
\]
preserves hyperbolicty if and only if
\[
W[Q_0,Q_2]^2-W[Q_0,Q_1]W[Q_1,Q_2]\le 0.
\]
\end{thm}

\begin{proof}
Since we are assuming $Q_0\ll Q_1 \ll Q_2$ then $c_0,c_1,c_2$ are of the same sign and $r_1\le r_3\le r_2$. Define,
\[
\barr{rcl}
w(x) & := & W[Q_0,Q_2]^2-W[Q_0,Q_1]W[Q_1,Q_2]\\
\\
&= &  c_0 c_2 (4 c_0 c_2-c_1^2 ) x^2 + 2 c_0 c_2 (-2 c_0 c_2 (r_1 + r_2) + c_1^2 r_3) x \\
\\
& & + c_0 c_2 (c_0 c_2 (r_1 + r_2)^2+c_1^2 (r_1 r_2 -r_1 r_3 - r_2 r_3)).\\
\earr
\]

Suppose $r_1=r_2$, then $w(x)=-c_0 c_2 (c_1^2-4c_0 c_2)(x-r_1)^2$. It is clear that $w(x)\le 0$ if and only if $0\le c_1^2-4 c_0 c_2$, thus we apply Theorem \ref{thm:quad-op'}. 

Suppose $0= c_1^2-4 c_0 c_2$ and $r_1\not= r_2$. By Lemma \ref{thm:quarter}, Theorem \ref{thm:quad-op} can restated as, ``$T$ is hyperbolicity preserving if and only if $2r_3=r_1+r_2$''. We recalculate $w$, under the assumption that $0=c_1^2-4 c_0 c_2$, 
\beq
w(x)= 4c_0^2c_2^2(2r_3-r_1-r_2)x +c_0^2c_2^2(2(r_1+r_2)(r_1+r_2-2r_3)-(r_1-r_2)^2).
\eeq
We now see that, $w(x)\le 0$, if and only if, $2r_3=r_1+r_2$. 

Thus we may assume $0\not=c_1^2-4c_0 c_2$ and $r_1\not= r_2$, in which case $w$ is a quadratic with vertex 
\beq
\lr{r_3,\frac{c_0c_1^2c_2}{c_1^2-4c_0c_2}\lr{c_0c_2(r_1-r_2)^2+c_1^2(r_1-r_3)(r_2-r_3)}}.
\eeq
Since $w$ is a quadratic then $w(x)\le 0$ if and only if the leading coefficient 
\beq
c_0c_2(4c_0c_1-c_1^2)<0
\eeq
and y-coordinate of the vertex 
\beq
\frac{c_0c_1^2c_2}{c_1^2-4c_0c_2}\lr{c_0c_2(r_1-r_2)^2+c_1^2(r_1-r_3)(r_2-r_3)}\le 0.
\eeq
Thus, we can say that $w(x)\le 0$ if and only if $0<c_1^2-4c_0c_1$ and $0\le c_1^2(r_1-r_3)(r_3-r_2)-c_0c_2(r_1-r_2)^2$. By Lemma \ref{thm:quarter} and Theorem \ref{thm:quad-op} those conditions are equivalent to $T$ preserving hyperbolicity. 
\end{proof}

It is unnecessary to assume that the polynomial coefficients of $T$ have real zeros, as this will natural follow from \ref{thm:bra-int}. Furthermore if $Q_2$ is a quadratic then Lemma \ref{thm:bra-int} states that $Q_1$ cannot be a non-zero constant if $T$ is to preserve hyperbolicity. To summarize we restate Theorem \ref{thm:quad-wron} with a little more generality. 

\begin{thmmain}
Suppose $Q_2,Q_1,Q_0$ are real polynomials such that $deg(Q_2)=2$, $deg(Q_1)\le 1$, $deg(Q_0)=0$. Then 
\begin{equation} \nonumber
T=Q_2D^2+Q_1D+Q_0
\end{equation}
preserves hyperbolicty if and only if
\begin{equation} \nonumber
W[Q_0,Q_2]^2-W[Q_0,Q_1]W[Q_1,Q_2]\le 0, \;\; \mbox{and} \;\; Q_0 \ll Q_1 \ll Q_2.
\end{equation}
\end{thmmain}

\section{Multiplier Sequences}

We now wish to establish several consequences of the above quadratic operators. 

\begin{defn}
Let $\{P_n\}$ be a basis for $\R[x]$. Let $\{A_n\}$ be a sequence of real numbers. If there is a linear operator, $T$, such that $T[P_n]=A_n P_n$ for every $n\in\N$, then we call $T$ a \emph{$P_n$-multiplier operator}. We will sometimes write $T=\{A_n\}$ when there is no question of the basis. Likewise, if there is a hyperbolicity preserver, $T$, such that $T[P_n]=A_n P_n$ for every $n\in\N$, then we call $T$ a \emph{$P_n$-multiplier squence} and sometimes write $T=\{A_n\}$. 
\end{defn}

There is a natural relationship between differential equations, differential operators, and $P_n$-multiplier sequences. This relationship has been used (\cite{bdfu,forgpio}), but never explicitly stated. 

\begin{thm}
Let $P_n$ be a basis for $\R[x]$. Suppose for each $n\in\N$, $P_n$ satisfies the differential equation
\[
\sum_{k=0}^\infty Q_k(x) y^{(k)}=A_n y
\]
where $\{Q_k\}$ is a sequence of real polynomials  and $\{A_n\}$ is a sequence of real numbers. Then $A_n$ is a $P_n$-multiplier sequence if and only if 
\[
\sum_{k=0}^\infty Q_k D^k
\]
is a hyperbolicity preserver. 
\end{thm}

Using Theorem \ref{thm:quad-op} we can restate the above theorem. 

\begin{thm}
Let $P_n$ be a simple set for $\R[x]$ and let $\{A_n\}$ be a sequence of real numbers. Let $c_i,r_j\in\R$, $i=0,1,2$, $j=1,2,3$, $c_2\not=0$. Suppose for each $n\in\N$ that $P_n$ satisfies the differential equation, 
\[
c_2(x-r_1)(x-r_2)y''+c_1(x-r_3)y'+c_0 y = A_n y
\]
then $\{A_n\}$ is a $P_n$-multiplier sequence, if and only if, $c_0,c_1,c_2$ are of the same sign and
\[
0\le c_1^2 \left (\frac{(r_1-r_3)(r_3-r_2)}{(r_2-r_1)^2}\right)-c_2c_0.
\]
In light of Theorem \ref{thm:quad-op'}, we take $\frac{(r_1-r_3)(r_3-r_2)}{(r_2-r_1)^2}=\frac{1}{4}$ in the case that $r_1=r_2$. 
\end{thm}

A large number of very well known bases for $\R[x]$ satisfy differential equations of the above form (\cite[pg. 173, 188, 204, 258]{rain}). We exhibit classes of multiplier sequences for Legendre, Jacobi, and the standard basis. We state the cooresponding differential equations:

Standard Basis
\begin{equation}
Ax^2(x^n)''+Bx(x^n)'+C (x^n) = (An(n-1)+Bn+C) x^n 
\end{equation}

Lengendre Polynomials
\begin{equation}
A(x^2-1)P_n''+2AxP_n'+BP_n = (An(n+1)+B)P_n 
\end{equation}



Jacobi Polynomials
\begin{equation}
\barr{c}
A(x^2-1)(P_n^{\alpha,\beta}) '' + A((\alpha+\beta+2)x-(\beta-\alpha))(P_n^{\alpha,\beta})' +B P_n ^{\alpha,\beta}\\
\ \hfill = (An(n+\alpha+\beta+1)+B) P_n ^{\alpha,\beta}  
\earr
\end{equation}

We now establish several classes of multiplier sequences. 

\begin{thm} 
Let $A,B,C\in\R$. Then $\{An(n-1)+Bn+C\}$ is a classic multiplier sequence if and only if $A,B,C$ are of the same sign and
\begin{equation}\nonumber
0\le B^2-4AC.
\end{equation} 
\end{thm}

\begin{thm} 
Let $A,B\in\R$, $A\not=0$. Then $\{An(n+1)+B\}$ is a $P_n$-multiplier sequence (Lengendre multiplier sequence) if and only if
\begin{equation}\nonumber
0\le \frac{B}{A} \le 1.
\end{equation} 
\end{thm}

\begin{thm} \label{jacquadms}
Let $A,B\in\R$, $A\not=0$. Then $\{An(n+\alpha+\beta+1)+B\}$ is a $P_n^{\alpha,\beta}$-multiplier sequence (Jacobi multiplier sequence) if and only if
\begin{equation}\nonumber
0\le \frac{B}{A} \le (\alpha+1)(\beta+1) \text{ and }-1\le \alpha,\beta.
\end{equation} 
\end{thm}

\begin{proof}
Notice, 
\begin{equation}
(A(\alpha+\beta+2))^2\left(1-\frac{\beta-\alpha}{\alpha+\beta+2}\right) \left(1+\frac{\beta-\alpha}{\alpha+\beta+2}\right)-4AB\ge 0 
\end{equation}
and 
\begin{equation}
A,A(\alpha+\beta+2),B \text{ are of the same sign}
\end{equation}
is equivalent to 
\begin{equation}
0\le \frac{B}{A} \le (\alpha+1)(\beta+1)\text{ and }-1\le\alpha,\beta.
\end{equation} 
\end{proof}

\end{document}